\documentclass[preprint,12pt]{elsarticle}

\usepackage{amssymb}
\usepackage{bbding}
\usepackage{geometry}
\usepackage{moreverb}

\usepackage{graphicx,epsfig}

\usepackage{multirow}
\usepackage{algorithm}
\usepackage{algorithmic}
\usepackage{float}
\usepackage{amsthm}
\usepackage{amsfonts}

\usepackage{color}
\usepackage{mathrsfs}
\usepackage{amsmath}
\usepackage{cases}
\usepackage{array}
\usepackage{booktabs}
\usepackage{CJK}
\usepackage{type1cm}
\usepackage{times}
\usepackage[marginal]{footmisc}
\usepackage{dsfont}
\newtheorem{definition}{Definition}
\newtheorem{theorem}{Theorem}
\newtheorem{lem}{Lemma}
\newtheorem{Assumption}{Assumption}
\newtheorem{Remark}{Remark}
\newtheorem{Proposition}{Proposition}

\newtheorem{Problem}{Problem}

\begin{document}

\begin{frontmatter}


 \tnotetext[label1]{This work was supported by the National Natural Science Foundation of China under Grant  No.62373229.}

\title{Model-free stochastic linear quadratic  design by semidefinite programming}

\author[label1,label2]{Jing Guo}
\author[label3]{Xiushan Jiang}
\author[label1]{Weihai Zhang\corref{cor1}}
\ead{w\_hzhang@163.com }
\cortext[cor1]{Corresponding author.}
 \address[label1]{College of Electrical Engineering and Automation, Shandong University of Science and Technology, 
             Qingdao,
            266590,
	            China}
\address[label2]{School of Mathematics and Statistics,
		Shandong University of Technology,
	             Zibo,
	             250049,
	            China}
\address[label3]{College of New Energy, China University of Petroleum (East China),
            	Qingdao,
            	266580,
            	China}


\begin{abstract}
In this article, we study a model-free design approach for stochastic linear quadratic (SLQ) controllers. Based on the convexity of the SLQ dual problem and the Karush-Kuhn-Tucker (KKT) conditions, we find the relationship between the optimal point of the dual problem and the Q-function, which can be used to develop a novel model-free semidefinite programming (SDP)  algorithm for deriving optimal control gain. This study provides a new optimization perspective for understanding Q-learning algorithms and lays a theoretical foundation for effective reinforcement learning (RL) algorithms. Finally, the effectiveness of the proposed model-free SDP algorithm is demonstrated by two case simulations.
\end{abstract}




\end{frontmatter}


\section{Introduction}
Reinforcement learning (RL), which has its roots  in the psychology of animal learning,  is an approach to learning by
 trial-and-error.
RL originally attracted a lot of attention from the artificial intelligence field, and Sutton established the link between RL and the control domain  \cite{sutton1999reinforcement}. 
In the control domain, RL approach means the interaction of the controller to the dynamic system, where the aim is to learn the optimal control policy by observing certain indices without full knowledge of the system dynamics \cite{sutton1999reinforcement, bertsekas2019reinforcement}. Over the past decades, many effective
model-free RL algorithms have been proposed for a variety of optimal control problems associated with various dynamic systems, see, for example, \cite{bian2016adaptive, kiumarsi2017h, pang2021robust, jiang2024online, wang2024online} and recent survey \cite{kiumarsi2017optimal} for
details. However, most of the current RL methods are weakly extendible or lacking in theoretical assurances, thus necessitating further investigation of RL.

In order to further comprehend the intrinsic mechanisms of various RL algorithms and also to better utilize the optimization theory to make more progress in the field of RL, many investigators  have started to reconsider
 the linear quadratic regulator (LQR) problem, which, as a long-standing and well known archetypal problem in optimal control problems, captures the characteristics of RL despite the simplicity of the problem setup. For time invariant systems, a standard approach to solving LQR problem is to solve the Bellman equation or the algebraic Riccati equation (ARE) via dynamic programming \cite{bertsekas2012dynamic}. Studies on the design of model-free controllers for LQR by means of the methods of RL, such as temporal difference \cite{sutton1988learning}, Q-learning \cite{watkins1992q}  and so  on, can be found in \cite{bradtke1994adaptive, rizvi2018output, krauth2019finite, lai2023model}. With the development of convex analysis \cite{boyd2004convex}, LQR problem, for which the system dynamics are completely known, has   been extensively studied using  convex optimization techniques \cite{boyd1994linear,rami2000linear, yao2001stochastic}, for example,  \cite{yao2001stochastic} studied the duality of the stochastic linear quadratic (SLQ) problem using  semidefinite programming (SDP). 
 Recently, many researchers have reconsidered the LQR problem from the point of   optimization  by connecting the RL method with classical convex optimization theory, and some valuable results have been given according to the duality of LQR problem, see, for example \cite{lee2018primal,li2022model,farjadnasab2022model}. Specifically, based on  the results in \cite{yao2001stochastic},  \cite{lee2018primal}  established a novel primal-dual Q-learning framework for LQR according to Lagrangian duality theories, utilized the framework to design a model-free primal-dual Q-learning algorithm to solve the LQR problem, and further interpreted the duality results from the Q-learning perspective.
 \cite{li2022model} studied the model-free design of SLQ controllers from the viewpoint of primal-dual optimization. The presented optimization formulation is different from that of \cite{lee2018primal} owing to the effect of random additive noise, and this difference  also brings some difficulties in the equivalence descriptions of the primal and dual  problems, respectively. In the work of both \cite{lee2018primal}  and \cite{li2022model}, proving that strong duality holds is a challenging task, and both of their proposed primal-dual Q-learning algorithms require that the initial policy  to be stable and requires the collection of a set of trajectories whose initial states are linearly independent. Unlike the work in \cite{lee2018primal},    \cite{farjadnasab2022model} exploited the properties of optimization frameworks and  proposed a new model-free algorithm on the basis of SDP for learning Q-function. The algorithm is characterized by high sampling efficiency, intrinsic robustness to model uncertainty, and does not need to provide an initially  stabilizing controller.
 
Inspired by the above works, this paper investigates the model-free SLQ design by SDP. The SLQ problem is first equivalently re-represented as a non-convex optimization problem. The nonconvexity of the primal problem leads to the nontrivial task of proving its strong duality, as done in \cite{li2022model}. Instead of proving strong duality, we start directly from the dual problem of the SLQ problem, use the KKT conditions and the convexity of the dual problem  to find out the relationship between the optimal point of the dual problem and the parameters of the Q-function, which can be directly used to  derive a new model-based SDP algorithm to obtain the controller gain. Finally, a model-free implementation of the controller design is given based on Monte-Carlo method.

 The proposed model-free  SDP approach has the following characteristics that are worth noting:

\begin{itemize}
	\item  It does not need to provide an initially  stabilizing  controller.
	\item  It does not need hyper-parameter adjustment.
	\item   The algorithm implementation procedure is done in a single step, not iteratively.
	\item   It only needs to collect state and input data over a short time horizon.	
\end{itemize} 

The remainder  is structured below. In Section 2 the SLQ problem of interest is presented with some preliminary preparations. In Section 3, the SLQ problem is reformulated as an optimization problem and a novel model-based SDP approach  for SLQ controller design is proposed based on the relationship between the dual problem and the Q-function. In Section 4, a model-free  implementation of model-based SDP approach  is presented. In Section 5, the proposed model-free algorithm is verified by applying it to different systems, which include both stable and unstable dynamic systems. Some conclusions are given in Section 6.

For convenience, we adopt the following notations:\par 
$\mathcal{Z}_{+} $: the collection  of nonnegative integers;  $\mathcal{R}^{l}$: the collection  of $l$ dimensional vectors; $\mathcal{R}^{l \times m}$: the collection  of $l \times m$ real matrices.  $I_{l}$: the  $l \times l$ identity matrix, $\pmb{0}$: the zero vector or matrix with the appropriate dimension. $M^{\top}$:  the transpose of a  vector or matrix $M$; $\operatorname{tr}\left( M\right)$: the trace of a square matrix $M$; $\rho\left( M\right)$: the spectral radius of  $M$; $\lambda_{\text {max}}\left(M \right)\left( \lambda_{\text {min}}\left(M \right)\right) $:  the maximum (minimum) eigenvalues of  $M$. $\mathbb{R}\left(M\right)$: the range space
of $M$;  $\mathbb{N}\left(M\right)$: the nullspace
of $M$.  For vectors $\pmb{x}$ and $\pmb{y}$, $\operatorname{col}\left(\pmb{x}, \pmb{y}\right)\triangleq\left[\begin{array}{c}\pmb{x}\\ \pmb{y}\end{array}\right]$. The direct sum of matrices $M_{1}$ and $ M_{2}$: $M_{1}\oplus M_{2} \triangleq\begin{bmatrix}
	M_{1}& \pmb{0}\\
	\pmb{0}&M_{2}
\end{bmatrix} $. $\mathcal{S}^{l}$, $\mathcal{S}^{l}_{+}$ and $\mathcal{S}^{l}_{++}$ are the collections of all $l \times l$ symmetric, symmetric positive semidefinite and symmetric positive definite matrices,  respectively; we write  $M \succeq \pmb{0}$, (resp. $M \succ \pmb{0}$) if $M \in \mathcal{S}^{l}_{+}$, (resp. $M \in \mathcal{S}^{l}_{++}$).

\section{Problem formulation and preliminaries}





Consider a class of discrete-time linear time-invariant systems subject to additive noise
 as:
\begin{equation}\label{eq1}
	\pmb{x}_{k+1}=A \pmb{x}_{k}+B \pmb{u}_{k}+\pmb{w}_{k}\,,
\end{equation}
where  $k \in \mathcal{Z}_{+}$ is the discrete-time index,
$\pmb{x}_{k} \in \mathcal{R}^{n}$ is
the system state, $\pmb{u}_{k} \in \mathcal{R}^{m}$ is
the control input, $\pmb{w}_{k} \in \mathcal{R}^{n}$ is
the system additive noise defined on a
given complete probability space $\left(\Omega, \mathscr{F},\mathcal{P}\right)$, $A \in \mathcal{R}^{n \times n}$ and $B \in \mathcal{R}^{n \times m}$ are the system coefficient matrices. The initial state $\pmb{x}_{0}$ is assumed to be a random variable. For convenience, it
is further assumed that the following Assumption 1 holds.
\begin{Assumption}\rm  \label{A1}
	\begin{itemize} 
		\item[(\romannumeral1)]	$\pmb{x}_{0} \sim \mathcal{N}_{n}( \pmb{\mu}_{0},\Sigma_{0})$, where  $\Sigma_{0}  \succ \pmb{0} \,;$	 
		\item[(\romannumeral2)]  The sequence of random vectors  $\left\{\pmb{w}_{k}, k \in \mathcal{Z}_{+} \right\}$ is independent and identically distributed, and furthermore,  $\pmb{w}_{k} \sim \mathcal{N}_{n}\left(\pmb{0}, \Sigma\right)$  for each $k \in \mathcal{Z}_{+}$,  where  $ \Sigma \succeq \pmb{0} \,; $	
		\item[(\romannumeral3)]  $\pmb{x}_{0}$ is independent of $\left\{\pmb{w}_{k}, k \in \mathcal{Z}_{+} \right\}\,.$
	\end{itemize}	 
\end{Assumption}

\begin{definition}\rm \cite{kubrusly1985mean}
	System~(\ref{eq1}) with $\pmb{u}_{k} \equiv \pmb{0}$ is
	said mean square stable (MSS) if  for
	any  $\pmb{x}_{0}$ and  $\left\{\pmb{w}_{k}, k \in \mathcal{Z}_{+} \right\}$
	satisfying Assumption~\ref{A1}, there exist $ \pmb{\mu} \in \mathcal{R}^{n}$ and $X \in \mathcal{R}^{n \times n}$,  which are independent of $\pmb{x}_{0}$,
	such that 
	\begin{itemize} 
		\item[(\romannumeral1)]$\lim _{\substack{{k \rightarrow \infty}}} \|\mathbb{E}\left( \pmb{x}_{k}\right) -\pmb{\mu}\|=0$,
		\item[(\romannumeral2)]$\lim _{\substack{{k \rightarrow \infty}}} \|\mathbb{E}\left( \pmb{x}_{k}\pmb{x}_{k}^{\top}\right) -X\|=0$.
	\end{itemize}	
\end{definition}
The following lemma gives  sufficient and necessary conditions under which the system is  MSS.
\begin{lem}\rm \label{lem1} \cite{kubrusly1985mean} \cite{lai2023model} \cite{li2022model}
	
	The following statements are equivalent:
	\begin{itemize} 
		\item[(\romannumeral1)]	System~(\ref{eq1}) with $\pmb{u}_{k} \equiv \pmb{0}$ is	MSS;	 
		\item[(\romannumeral2)] $\rho\left( A\right)<1 $;
		\item[(\romannumeral3)] when $r > 1 - \frac{\lambda_{\text {min}}\left(Z \right) }{\lambda_{\text {max}}\left(A^{\top}SA \right)}$, for each given $Z \succ \pmb{0}$, there exists a unique
		$S \succ \pmb{0}$, such that $r A^{\top}SA+Z=S$.
	\end{itemize}	
\end{lem}
In view of Lemma~\ref{lem1}, we assume  that 	$r \in (0,1) $ is close to 1 throughout the paper.

\begin{definition}\rm 
	The state feedback control  $\pmb{u}_{k}=L\pmb{x}_{k}$ is called admissible, if  $	\pmb{x}_{k+1}=\left(A+BL\right) \pmb{x}_{k}+\pmb{w}_{k}$ is MSS.  In this case,   $L \in \mathcal{R}^{m \times n} $ is called a
	stabilizing gain of system~(\ref{eq1}), and the set of all stabilizing  gains is denoted as $\mathcal{L}$.  		
\end{definition}

Consequently, the cost functional can be
defined as 
\begin{equation}\label{j1}
	J(L,\pmb{x}_{0})\triangleq\sum_{k=0}^{\infty}r^{k} \mathbb{E}\left[\begin{array}{l}
		\pmb{x}_{k} \\
		L\pmb{x}_{k}
	\end{array}\right]^{\top}W\left[\begin{array}{l}
		\pmb{x}_{k} \\
		L\pmb{x}_{k}
	\end{array}\right]\,,
\end{equation}
where $r \in (0,1) $ is the discount factor and $W \triangleq 
	Q\oplus R \succeq \pmb{0}$ is a block-diagonal matrix   
consisting of the  weighting matrices $Q\in \mathcal{S}^{n}_{+}$ and
$R \in \mathcal{S}^{m}_{++}$, respectively. 

In this
paper, we consider the infinite-horizon SLQ problem. 
\begin{Problem}\rm ( SLQ problem )\label{p1}\par	
	Find an optimal feedback gain $L^{*} \in \mathcal{L}$, if it exists,  that  minimize the cost functional~(\ref{j1}). That is, 
	\begin{equation*}
		L^{*} \triangleq
		\underset{L \in \mathcal{L}}{\arg \min }\ J(L,\pmb{x}_{0})\,.
	\end{equation*}  		
\end{Problem} 

The following assumptions are necessary to ensure that an optimal feedback gain $L^{*}$ exists.
\begin{Assumption}\rm \label{A2}
	Assume that 
	\begin{itemize} 
		\item[(\romannumeral1)]	$(A, B)$ is stabilizable;
		\item[(\romannumeral2)]  $(A, \sqrt{Q} )$ is detectable.
	\end{itemize}
\end{Assumption}
Under Assumption~\ref{A2},    the optimal value of $\underset{L \in \mathcal{L}}{ \inf }\ J(L,\pmb{x}_{0})$ exists
and is attained, and the corresponding optimal cost $J\left(L^{*},\pmb{x}_{0}\right) $, which is  abbreviated as $ J^{*}$, is given by \cite{lai2023model} 	
\begin{equation*}
	J^{*}=\mathbb{E}\left( \pmb{x}_{0}^{\top} P^{*} \pmb{x}_{0} \right) +\frac{r}{1-r}\operatorname{tr}\left(P^{*} \Sigma \right)\,,
\end{equation*}
where $P^{*}\succeq \pmb{0}$ is the unique solution to the following discrete-time algebraic Riccati equation (DARE):
\begin{equation}\label{are1}
	P= Q+r A^{\top} PA-r^{2}	A^{\top} P B \left( 
	R+r B^{\top} PB\right) ^{-1}	B^{\top} P A\,.
\end{equation}  
The  corresponding optimal control gain $L^{*}$ for the SLQ problem is given by
\begin{equation*}
	L^{*}=-r \left( 
	R+r B^{\top} P^{*}B\right) ^{-1}	B^{\top} P^{*} A\,.
\end{equation*}

From the above, it is easy to conclude that the SLQ problem can be solved with knowing  the system dynamics $A$ and $B$.
The following Q-learning provides a model-free design approach to solve the SLQ problem.
Define the Q-function as (see \cite{lai2023model} for more details)
\begin{equation*}
	Q\left(\pmb{x}_{k},\pmb{u}_{k}\right) \triangleq \mathbb{E}\left(U(\pmb{x}_{k}, \pmb{u}_{k})\right) +r		V\left(L,\pmb{x}_{k+1}\right)\,,
\end{equation*}
where the value function
\begin{equation*}
	V\left(L,\pmb{x}_{k}\right) \triangleq \mathbb{E}\sum_{i=k}^{\infty}r^{i-k} U (\pmb{x}_{i}, L\pmb{x}_{i})\,,
\end{equation*}
with $	U(\pmb{x}_{i}, \pmb{u}_{i})\triangleq \pmb{x}_{i}^{\top} Q  \pmb{x}_{i}+\pmb{u}_{i}^{\top} R  \pmb{u}_{i}$.

The optimal value function produced by  the optimal admissible control policy is denoted  as 
\begin{equation*}
	V^{*}\left(\pmb{x}_{k}\right)\triangleq V\left(L^{*},\pmb{x}_{k}\right)=\min _{L \in \mathcal{L}}  V\left(L,\pmb{x}_{k}\right)\,.
\end{equation*}
For the case of SLQ, we have
\begin{lem}\rm \label{lem3} \cite{lai2023model}
	Define  the optimal
	Q-function 	as $Q^{*}\left(\pmb{x}_{k},\pmb{u}_{k}\right) \triangleq \mathbb{E}\left(U(\pmb{x}_{k}, \pmb{u}_{k})\right) +r	V^{*}\left(\pmb{x}_{k+1}\right) $, then  $Q^{*}\left(\pmb{x}_{k},\pmb{u}_{k}\right) $ can be expressed as 
	\begin{equation}\label{H}
		Q^{*}\left(\pmb{x}_{k},\pmb{u}_{k}\right)= \mathbb{E}\left[  \left[ \pmb{x}_{k}^{\top},\pmb{u}_{k}^{\top}\right] H^{*}\left[ \pmb{x}_{k}^{\top},\pmb{u}_{k}^{\top}\right]^{\top}\right]  +\frac{r}{1-r}\operatorname{tr}\left(P^{*} \Sigma \right)\,,
	\end{equation}
	where $P^{*}\succeq \pmb{0}$ is the unique solution to DARE~(\ref{are1}), $H^{*}\triangleq\begin{bmatrix}
			H^{*}_{11}&H^{*}_{12}\\
			H^{*}_{21}&H^{*}_{22}
		\end{bmatrix}$
	with $H^{*}_{11}=Q+r A^{\top} P^{*}A$, $H^{*}_{12}=r A^{\top} P^{*} B=\left( H^{*}_{21}\right) ^{\top}$, $H^{*}_{22}=R+r B^{\top} P^{*} B$.	
	Furthermore, the optimal control is derived from
	\begin{equation*}
		u^{*}_{k}  =L^{*}\pmb{x}_{k}=\underset{\pmb{u}_{k}}{\arg \min }\ Q^{*}(\pmb{x}_{k}, \pmb{u}_{k}) 
	\end{equation*}
	with the optimal control gain
	\begin{equation} \label{lstar}
		L^{*}=-\left( H^{*}_{22}\right) ^{-1} \left( H^{*}_{12}\right) ^{\top}\,.	
	\end{equation}
\end{lem}

\section{ Model-based design of  SLQ control via SDP }

In Section 3.1, the SLQ problem is first reformulated as a non-convex optimization problem, and then in Section 3.2, a non-iterative SDP optimization method is  developed to find the optimal Q-function parameters from the dual of the SLQ problem. Finding the optimal Q-function means finding the optimal control policy for the SLQ problem.
\subsection{Problem reformulation} 
By introducing the augmented system,  which is described by
\begin{equation}
	\pmb{z}_{k+1}=A_{L} \pmb{z}_{k}+\bar{L} \pmb{w}_{k} \,,\label{au}
\end{equation} 
where the augmented state vector $\pmb{z}_{k} \triangleq \operatorname{col}\left( \pmb{x}_{k},\pmb{u}_{k}\right) $ with the initial augmented state vector $\pmb{z}_{0} \triangleq \operatorname{col}\left(\pmb{x}_{0},L\pmb{x}_{0}\right) $, $A_{L}\triangleq \left[\begin{array}{cc}A & B \\ L A & L B\end{array}\right]  $, $\bar{L} \triangleq \left[\begin{array}{c}I_{n} \\L\end{array}\right]  $, we can obtain  the alternative formulation of the Problem~\ref{p1}  (SLQ problem) and described
by the following Problem~\ref{p2}.

\begin{Problem}\rm \label{p2} (Primal Problem):
	Non-convex optimization with optimization variables $L\in \mathcal{R}^{m \times n}$ and $S \in \mathcal{S}^{n+m}$. 
	\begin{align}
		J_p \triangleq & \ \underset{L\in \mathcal{R}^{m \times n}\,,\ S \in \mathcal{S}^{n+m} }{\min} \ \operatorname{tr}(W S)\,, \notag\\
		\text { s.t. } \ & S \succ \pmb{0},\notag\\
		&r A_{L} S A_{L}^{\top}+\bar{L} X_{0}\bar{L}^{\top}+\frac{r}{1-r}\bar{L}\Sigma \bar{L}^{\top}=S \,.  \label{eqp32}	
	\end{align}	
\end{Problem}
Note that Problem~\ref{p2} is a non-convex optimization problem because constraint~(\ref{eqp32}) is not linear. We first prove that Problem~\ref{p2} is a non-convex optimization equivalent to Problem~\ref{p1}.
\begin{Proposition}\rm
	Problem~\ref{p2} is
	equivalent to Problem~\ref{p1} in a meaning that $J_p =J^{*}$ and $L_p=L^{*}$,  where  $(S_p, L_p)$ is the optimal point of Problem~\ref{p2}. In addition,	 $(S_p, L_p)$ is unique.
\end{Proposition}
\begin{proof}
	Using the properties of matrix traces, we can reformulate the objective function of Problem 1 as 
	\begin{equation*}
		J\left(L, \pmb{x}_{0}\right)=	 \operatorname{tr}(W S)\,, 
	\end{equation*}
	where 
	\begin{equation} \label{rs}
		S \triangleq  \sum_{k=0}^{\infty}r^{k} \mathbb{E}\left[\begin{array}{l}
			\pmb{x}_{k} \\
			L	\pmb{x}_{k}
		\end{array}\right]\left[\begin{array}{l}
			\pmb{x}_{k} \\
			L	\pmb{x}_{k}
		\end{array}\right]^{\top}\,.	
	\end{equation} 
	From ~(\ref{au}), 	for $ k \geq 1$, we get 
	\begin{equation} \label{asv}
		\pmb{z}_{k}=A_{L}^{k} \pmb{z}_{0}+\sum_{j=0}^{k-1} A_{L}^{j} \bar{L} \pmb{w}_{k-j-1}\,.		
	\end{equation}
	Using Eq.~(\ref{asv}) and the item (\romannumeral2) in Assumption~\ref{A1},  for  $k \geq 1$ and any $L \in \mathcal{L}$, one can obtain that 
	\begin{equation*}
		\begin{aligned}
			& \mathbb{E}\left[\pmb{z}_{k} \pmb{z}_{k}^{\top}\right] \\
			=& \mathbb{E}\left[A_{L}^{k} \pmb{z}_{0} \left( \pmb{z}_{0}\right) ^{\top}\left(A_{L}^{k}\right)^{\top}\right] +\sum_{j=0}^{k-1} \mathbb{E}\left[ A_{L}^{j} \bar{L} \pmb{w}_{k-j-1} \pmb{w}_{k-j-1}^{\top} \bar{L}^{\top}\left(A_{L}^{j}\right)^{\top}\right]\\
			=& A_{L}^{k} \bar{L} \mathbb{E}\left[\pmb{x}_{0} \pmb{x}_{0}^{\top}\right] \bar{L}^{\top}\left(A_{L}^{k}\right)^{\top}+\sum_{j=0}^{k-1} A_{L}^{j} \bar{L} \Sigma \bar{L}^{\top}\left(A_{L}^{j}\right)^{\top}\,.
		\end{aligned}
	\end{equation*}	
	Denote $X_{0} \triangleq\mathbb{E}\left[\pmb{x}_{0} \pmb{x}_{0}^{\top}\right]=\Sigma_{0}+\pmb{\mu}_{0}\pmb{\mu}_{0}^{\top}$ in the sequel for brevity, thus, $S$ in~(\ref{rs}) becomes
	\begin{equation*}		
		S =\sum_{k=0}^{\infty} r^{k} A_{L}^{k} \bar{L}X_{0}\bar{L}^{\top}
		\left(A_{L}^{k}\right)^{\top} +\sum_{k=1}^{\infty} r^{k}\sum_{j=0}^{k-1}  A_{L}^{j} \bar{L} \Sigma \bar{L}^{\top}\left(A_{L}^{j}\right)^{\top}\,.			
	\end{equation*}
	By an algebraic manipulation, we can get that
	\begin{equation*}
		S-r A_{L} S A_{L}^{\top}=\bar{L} X_{0} \bar{L}^{\top}+\sum_{k=1}^{\infty} r ^{k} \bar{L}\Sigma \bar{L}^{\top} =\bar{L} X_{0} \bar{L}^{\top}+\frac{r}{1-r} \bar{L}\Sigma \bar{L}^{\top}\,,
	\end{equation*}
	which means that  $S$  satisfies~(\ref{eqp32}).
	
	From  $X_{0}\succ \pmb{0}  $, $ \Sigma \succeq \pmb{0}$ and  $0<r<10$, we get $X_{0} +\frac{r}{1-r} \Sigma \succ \pmb{0}  $, which leads to  $\bar{L} \left( X_{0} +\frac{r}{1-r} \Sigma \right) \bar{L}^{\top}\succ \pmb{0}  $. Then, when  $r>1-\frac{\lambda_{\min }\left(\bar{L} \left( X_{0} +\frac{r}{1-r} \Sigma \right) \bar{L}^{\top}\right)}{\lambda_{\max }\left(A_{L}^{\top} S A_{L}\right)}$,  it follows from Lemma~\ref{lem1}  that  $\rho\left(A_{L}\right)<1$ if and only if~(\ref{eqp32}) has a unique  solution $S \succ \pmb{0}$. We can obtain  from  Lemma~1 in \cite{lee2018primal} that  $\rho(A+B L)=\rho\left(A_{L}\right)$. Thus, the constraint $L \in \mathcal{L}$  in Problem~\ref{p1} can be replaced by $S \succ \pmb{0}$  that does not change its optimal point.

	If  $\left(S_{p}, L_{p}\right)$  is the optimal point of Problem~\ref{p2} and $J_{p}$  is  the related optimal value, then  $ S_{p} \succ \pmb{0}$  and  $\rho\left(A_{L_{p}}\right)<1$.  $ L_{p} \in \mathcal{L}$  is a feasible point of Problem~\ref{p1}, implying   $J_{p} \geq J^{*}$. If  $S_{p} $ is the unique solution of~(\ref{eqp32}) with  $L_{p}=L^{*} \in \mathcal{L}$,  then the corresponding objective function of  Problem~\ref{p2} is  $J_{p}=J^{*}$. Therefore,  one can obtain that $J_{p}=J^{*}$. The uniqueness of $\left(S_{p}, L_{p}\right)$  is demonstrated by the uniqueness of  $L^{*}$.
\end{proof}
\begin{Remark}\rm
	Although both   \cite{li2022model} and this paper investigate  SLQ problems with additive noise, the cost functional in  \cite{li2022model} is defined from a   sample from the initial state, whereas the cost functional in this paper is defined directly from the assumption that the initial state is a random variable,  and thus the cost functional defined in this paper is more general. Because of this, the equivalent optimization formulation of SLQ problem obtained in this paper is different from that of \cite{li2022model}.
\end{Remark}
\subsection{Model-based  design of  SLQ control via SDP} 	

In this section, we first represent the Lagrange dual problem associated with Problem~\ref{p2} as an SDP problem, where the affine objective function shows the dependence of the objective function on the initial state and additive noise, while the  LMI constraints are independent of the initial state and additive noise, which also shows the independence of the optimal point of the dual problem from the initial state and the additive noise.  In addition, the optimal point of the dual problem is related to the parameters of the optimal Q-function, which provides us with a way to find the parameters of the optimal Q-function. 

 We begin by describing the Lagrangian dual problem associated with Problem~\ref{p2} as a standard convex optimization problem.
\begin{theorem}\rm  
	The Lagrange  dual problem  associated with Problem~\ref{p2} can be expressed as a convex optimization problem indicated
	 by Problem~\ref{p3}.
	\begin{Problem}\rm	\label{p3}  (Dual Problem):
		Convex optimization with  optimization variables $M \in \mathcal{R}^{n \times n}$ and $F =\begin{bmatrix}
			F_{11}&F_{12}\\
			F_{12}^{\top}&F_{22}
		\end{bmatrix} \in \mathcal{S}^{n+m}$ with $F_{11}\in \mathcal{S}^{n}$, $F_{22}\in \mathcal{S}^{m}$,  and  $F_{12} \in \mathcal{R}^{n \times m}$. 	
		\begin{align}
			\underset{F,M }{\max} & \ \operatorname{tr}\left( \left( X_{0}+\frac{r}{1-r} \Sigma\right) M\right), \label{dp1}\\
			\text { s.t. } 	&r \left[
			A \quad  B	\right]^{\top}\left(F_{11}-F_{12} F_{22}^{-1} F_{12}^{\top}\right)\left[
			A \quad  B	\right]-F+W\succeq \pmb{0}\,, \label{dp1-2}\\
			&F_{22}\succ \pmb{0}\,, \label{dp1-3} \\
			& F_{11}-F_{12} F_{22}^{-1} F_{12}^{\top}-M\succeq \pmb{0}\,.\label{dp1-1}
		\end{align} 
	\end{Problem}
\end{theorem}
\begin{proof} 
In order to express the Lagrangian, the multipliers $F_{0}\in \mathcal{S}^{n+m}_+$ and $F\in \mathcal{S}^{n+m}$  are introduced, and the Lagrangian associated with Problem~\ref{p2} is obtained as	
		\begin{equation*}
		\begin{aligned}
			\mathscr{L}_{1}(S,L, F_{0}, F) \triangleq &\operatorname{tr}(W S)-\operatorname{tr}(F_{0} S) \\
			&+\operatorname{tr}\left(\left(r A_{L} S A_{L}^{\top}+\bar{L} X_{0}\bar{L}^{\top}+\frac{r}{1-r} \bar{L} \Sigma \bar{L}^{\top}-S\right) F\right) \\
			=&\operatorname{tr}\left(\left(r A_{L}^{\top} F A_{L}-F-F_{0}+W\right) S\right) \\
			&+\operatorname{tr}\left(\left( X_{0}+\frac{r}{1-r} \Sigma\right) \bar{L} ^{\top} F\bar{L}\right)\,.	
		\end{aligned}			
	\end{equation*}
	 The Lagrange dual function  is
	\begin{equation}\label{ld1}
		\mathcal{D}_{1}(F_{0}, F)\triangleq \underset{L} {\inf} \ \underset{S\succeq \pmb{0}} {\inf}  \mathscr{L}_{1}(S,L,F_{0}, F)\,,	
	\end{equation}
	and then, the Lagrange  dual problem  is 
	\begin{align}
		\underset{F_{0}, F }{\max} \ &  \mathcal{D}_{1}(F_{0}, F)\,,\label{d1}\\
		\text { s.t. } & 	F_{0}\succeq \pmb{0} \,.\label{d1-2}
	\end{align} 
	When  the Lagrangian  $ \mathscr{L}_{1}(S,L,F_{0}, F)$  is not bounded from below, then the value of the dual function $\mathcal{D}_{1}(F_{0}, F)$ is $-\infty$. Thus for the Lagrange dual problem~(\ref{d1}), the dual feasibility requires that $F_{0} \succeq \pmb{0}$ and that $\mathcal{D}_{1}(F_{0}, F)>-\infty$.  For given $F_{0}$ and $F$, since $S\succeq \pmb{0}$, the Lagrange dual function  $\mathcal{D}_{1}(F_{0}, F)$ which is given in~(\ref{ld1}) becomes
	\begin{equation}\label{df1}
		\begin{array}{l}
			\mathcal{D}_{1}(F_{0}, F)=  \left\{\begin{array}{ll}
				\underset{L  }{\inf}  \operatorname{tr}\left(\left( X_{0}+\frac{r}{1-r} \Sigma\right) \bar{L} ^{\top} F\bar{L}\right)\,, & \text { if } r A_{L}^{\top} F A_{L}-F-F_{0}+W \succeq \pmb{0}\\
				-\infty\,, & \text { otherwise. }
			\end{array}\right.
		\end{array}	
	\end{equation}
	Since the Lagrange multiplier $F_{0}$ only appears in the constraint condition of the Lagrange dual function~(\ref{df1}): $r A_{L}^{\top} F A_{L}-F-F_{0}+W \succeq \pmb{0}$, the dual feasibility is equivalent to	
	$r A_{L}^{\top} F A_{L}-F+W \succeq \pmb{0}$. Using the block-diagonal matrix representation of the Lagrange multiplier $F $ in Problem~\ref{p3}, the objective function  in~(\ref{df1}), which
	is denoted by $\mathscr{L}_{2}(L, F)$, can be represented as
	\begin{equation}\label{eq14}
		\begin{aligned}
			\mathscr{L}_{2}(L, F) \triangleq& \operatorname{tr}\left(\left( X_{0}+\frac{r}{1-r} \Sigma\right) \bar{L} ^{\top} F\bar{L}\right)\\
			= &\operatorname{tr} \left(\left( X_{0}+\frac{r}{1-r} \Sigma\right) \left(L^{\top} F_{22} L+L^{\top} F_{12}^{\top}+F_{12} L+F_{11}\right)\right)\,. 			
		\end{aligned}		 	
	\end{equation}	 
	Therefore, the dual problem~(\ref{d1})--(\ref{d1-2}) is equivalent to
	\begin{align}
		\underset{ F}{\max} \ & \underset{ L } {\inf}   \mathscr{L}_{2}(L, F)\,, \label{dp2-1}\\
		\text { s.t. } & 	r A_{L}^{\top} F A_{L}-F+W \succeq \pmb{0} \,. \label{dp2-2}
	\end{align}

	To give an explicit expression for the function   $\mathcal{D}_2 ( F)\triangleq \underset{ L } {\inf}   \mathscr{L}_{2}(L, F)$,  consider the following three different situations: 
	
	Situation 1: Let $ F_{22}\succ \pmb{0}$. In this situation, letting $\frac{\partial \mathscr{L}_{2}(L, F)}{\partial L}=\pmb{0}$ yields	
	\begin{equation*}
		\begin{array}{l}
			F_{22} L \left( X_{0}+\frac{r}{1-r} \Sigma\right) +F_{22}^{\top} L \left( X_{0}+\frac{r}{1-r} \Sigma\right) ^{\top}+F_{12}^{\top} \left( X_{0}+\frac{r}{1-r} \Sigma\right) +F_{12}^{\top} \left( X_{0}+\frac{r}{1-r} \Sigma\right) ^{\top}=\pmb{0}\,. 
		\end{array}		
	\end{equation*}	
	By solving the above equation, the optimal point of the function $\mathscr{L}_{2}(L, F)$  can be solved to be $L^{*}=-F_{22}^{-1} F_{12}^{\top}$, so the  infimum of $\mathscr{L}_{2}(L, F)$ is attained at $L^{*}$,  and hence 
	\begin{equation}\label{oj2}
		\mathcal{D}_2 ( F) =   \mathscr{L}_{2}(L^{*}, F)=\operatorname{tr} \left(\left( X_{0}+\frac{r}{1-r} \Sigma\right) \left(F_{11}-F_{12} F_{22}^{-1} F_{12}^{\top}\right)\right)\,.
	\end{equation}
	Substituting $L^{*}$ in the constraint~(\ref{dp2-2})  and combining  with the objective function in~(\ref{oj2}) results in an alternative equivalent representation of the dual problem~(\ref{dp2-1})--(\ref{dp2-2}):
	\begin{align*}
		\underset{F }{\max} &\ \operatorname{tr} \left(\left( X_{0}+\frac{r}{1-r} \Sigma\right) \left(F_{11}-F_{12} F_{22}^{-1} F_{12}^{\top}\right)\right),\\
		\text { s.t. } & 	r \left[\begin{array}{ll}
			A & B
		\end{array}\right]^{\top}\left(F_{11}-F_{12} F_{22}^{-1} F_{12}^{\top}\right)\left[\begin{array}{ll}
			A & B
		\end{array}\right]-F+W \succeq \pmb{0}\,.
	\end{align*}	
	
Situation 2: Let $F_{22} \succeq \pmb{0} $ be a singular matrix assuming that $ \mathbb{R}\left(F_{22}\right)\supseteq  \mathbb{R}\left(F_{12}^{\top}\right) $. In this situation, the infimum of $\mathscr{L}_{2}(L, F)$ still exists. Suppose $\operatorname{rank}\left( F_{22}\right)  = q<m $, then $F_{22}$ can be factored as $F_{22}=U_{q} \Lambda_{q} U_{q}^{\top}$,	 where $ U_{q}\in \mathcal{R}^{m \times q}$ satisfies $ U_{q}^{\top}U_{q} =I_{q}$ and $\Lambda_{q} =
	\operatorname{diag}(\lambda_1,\cdots,\lambda_q)$ with $\lambda_1\geq \lambda_2 \geq\cdots\geq\lambda_q>0$.
	Moreover, the columns of  $ U_{q} $ span $\mathbb{R}\left(F_{22}\right)$. The pseudo-inverse $F_{22}^{\dagger}$ can be expressed as $F_{22}^{\dagger}=U_{q} \Lambda_{q}^{-1} U_{q}^{\top}$. By a process similar to that of Situation~1, the Lagrangian  dual problem associated with problem~\ref{p2} for Situation~2 can be deduced as follows
		\begin{align*}
		\underset{F }{\max} &\ \operatorname{tr} \left(\left( X_{0}+\frac{r}{1-r} \Sigma\right) \left(F_{11}-F_{12} F_{22}^{\dagger} F_{12}^{\top}\right)\right),\\
		\text { s.t. } & 	r \left[\begin{array}{ll}
			A & B
		\end{array}\right]^{\top}\left(F_{11}-F_{12} F_{22}^{\dagger} F_{12}^{\top}\right)\left[\begin{array}{ll}
			A & B
		\end{array}\right]-F+W \succeq \pmb{0}\,.
	\end{align*}				
	Since $ \mathbb{R}\left(F_{22}\right)\supseteq  \mathbb{R}\left(F_{12}^{\top}\right) $, there exists   an unitary matrix $U_{m-q}$ such that the columns of $U_{m-q}$ span  $\mathbb{N}\left(F_{12}\right)$. Let   $\Lambda_{m-q}$  be any diagonal matrix where the diagonal elements are positive, and construct a matrix $\bar{F}_{22}\succ \pmb{0}$ as
	\begin{equation*}
		\bar{F}_{22}=\left[\begin{array}{ll}
			U_{q} & U_{m-q}
		\end{array}\right]\left[\begin{array}{cc}
			\Lambda_{q} & \pmb{0} \\
			\pmb{0} & \Lambda_{m-q}
		\end{array}\right]\left[\begin{array}{ll}
			U_{q} & U_{m-q}
		\end{array}\right]^{\top} \,.
	\end{equation*}
	Noting  that
	\begin{equation*}
		\begin{array}{l}
			F_{12} \bar{F}_{22}^{-1} F_{12}^{\top} \\
			=F_{12}\left[\begin{array}{ll}
				U_{q} & U_{m-q}
			\end{array}\right]\left[\begin{array}{cc}
				\Lambda_{q}^{-1} & \pmb{0} \\
				\pmb{0} & \Lambda_{m-q}^{-1}
			\end{array}\right]\left[\begin{array}{c}
				U_{q}^{\top} \\
				U_{m-q}^{\top}
			\end{array}\right] F_{12}^{\top} \\
			=F_{12} U_{q}\Lambda_{q}^{-1} U_{q}^{\top} F_{12}^{\top}-\underbrace{F_{12} U_{m-q} \Lambda_{m-q}^{-1} U_{m-q}^{\top} F_{12}^{\top}}_{0}\\
			=F_{12} F_{22}^{\dagger} F_{12}^{\top}\,,
		\end{array}	
	\end{equation*}		
	it can be concluded that each admissible  candidate $F_{22} \succeq \pmb{0}$ leads to a value of the objective which can be obtained by substituting a  counterpart  $\bar{F}_{22}\succ\pmb{0}$.
	
	Situation 3:  $F_{22} \succeq \pmb{0}$  while   $ \mathbb{R}\left(F_{22}\right)\nsupseteq  \mathbb{R}\left(F_{12}^{\top}\right) $ or at least one negative eigenvalue is contained in matrix $ F_{22}$ . In this situation,  the Lagrangian is unbounded from below through minimizing (\ref{eq14}) in terms of $L$.

To summarize, without loss of generality, we follow up with the assumption that
  $F_{22}\succ \pmb{0}$.  Therefore, the dual problem~(\ref{dp2-1}) -- (\ref{dp2-2}) is equivalent to
	\begin{align}
		\underset{F }{\max} &\ \operatorname{tr} \left(\left( X_{0}+\frac{r}{1-r} \Sigma\right) \left(F_{11}-F_{12} F_{22}^{-1} F_{12}^{\top}\right)\right)\,,\label{dp3-1}\\
		\text { s.t. } & 	r \left[\begin{array}{ll}
			A & B
		\end{array}\right]^{\top}\left(F_{11}-F_{12} F_{22}^{-1} F_{12}^{\top}\right)\left[\begin{array}{ll}
			A & B
		\end{array}\right]-F+W \succeq \pmb{0}\,,\label{dp3-2} \\
		& F_{22}\succ \pmb{0}\,. \label{dp3-3}
	\end{align}

	Finally, through the introduction of  the slack variable $ M \in \mathcal{S}^{n}$ and the added constraint~(\ref{dp1-1}),	
	the Lagrange  dual problem~(\ref{dp3-1})--(\ref{dp3-3}) can be equivalently described by Problem~\ref{p3}.	In fact, let    $F^{*}$  be the optimal  point of the dual problem~(\ref{dp3-1})--(\ref{dp3-3})  and $\left(F^{*},M^{*} \right) $  be the optimal  point of Problem~\ref{p3}. Based on the positive definiteness of $X_{0}+\frac{r}{1-r} \Sigma$ and 
	 the constraint~(\ref{dp1-1}), 
	\begin{equation*}
		\operatorname{tr}  \left(\left( X_{0}+\frac{r}{1-r} \Sigma\right)  M^{*}\right) \leq \operatorname{tr}\left(\left( X_{0}+\frac{r}{1-r} \Sigma\right) \left(F_{11}^{*}-F_{12}^{*}\left(F_{22}^{*}\right)^{-1}\left(F_{12}^{*}\right)^{\top}\right)\right)\,.	
	\end{equation*}
	
	Because $M^{*}$ is the maximum point of the objective function~(\ref{dp1-1}), 
	\begin{equation*}
		\operatorname{tr}\left(\left( X_{0}+\frac{r}{1-r} \Sigma\right)  M^{*}\right) \geq \operatorname{tr}\left(\left( X_{0}+\frac{r}{1-r} \Sigma\right) \left(F_{11}^{*}-F_{12}^{*}\left(F_{22}^{*}\right)^{-1}\left(F_{12}^{*}\right)^{\top}\right)\right)\,,	
	\end{equation*}
	and hence
	\begin{equation*}
		\operatorname{tr}  \left(\left( X_{0}+\frac{r}{1-r} \Sigma\right) \left(M^{*}-F_{11}^{*}+F_{12}^{*}\left(F_{22}^{*}\right)^{-1}\left(F_{12}^{*}\right)^{\top}\right)\right)=0\,.
	\end{equation*} 
	Therefore, all eigenvalues of the matrix $M^{*}-F_{11}^{*}+F_{12}^{*}\left(F_{22}^{*}\right)^{-1}\left(F_{12}^{*}\right)^{\top}$ are zero. As  $M^{*}$ is symmetric, it results that  
	\begin{equation*}
		M^{*}=F_{11}^{*}-F_{12}^{*}\left(F_{22}^{*}\right)^{-1}\left(F_{12}^{*}\right)^{\top} \,,
	\end{equation*}	
	which implies that the dual problem~(\ref{dp3-1})--(\ref{dp3-3}) is equivalent to  Problem~\ref{p3}.	
\end{proof}
The next theorem gives  the  relationship between the optimal point to the dual problem associated with SLQ problem  and  the optimal Q-function. 

\begin{theorem}\rm	\label{th2}
	Denote $\left(F^{*},M^{*} \right) $ as the optimal point to Problem~\ref{p3}, then 
	 $\left(F^{*},M^{*} \right) $  is independent of  $ X_{0}+\frac{r}{1-r} \Sigma$. Furthermore, $F^{*}=H^{*}$ and $M^{*}=P^{*}$,  where $ H^{*}$  is the  parameter of the optimal Q-function  expressed in~(\ref{H}) and  $P^{*}$  is the solution to  DARE~(\ref{are1}).
\end{theorem}
\begin{proof}
By introducing multipliers 	$G_{1}\in \mathcal{S}^{n+m}_+$, $G_{2}\in \mathcal{S}^{m}_+$ and $G_{3}\in \mathcal{S}^{n}_+$,  the Lagrangian associated with Problem~\ref{p3} is derived by
	\begin{equation*}
			\begin{aligned}
			&\mathscr{L}_{3}\left(F, M, G_{1}, G_{2}, G_{3}\right)\\
			=&-\operatorname{tr}\left( \left( X_{0}+\frac{r}{1-r} \Sigma\right)  M\right)  \\			
			&-\operatorname{tr}\left(G_{1} \left(r \left[\begin{array}{ll}
				A & B
			\end{array}\right]^{\top}\left(F_{11}-F_{12} F_{22}^{-1} F_{12}^{\top}\right)\left[\begin{array}{ll}
				A & B
			\end{array}\right] -F+W\right) \right) \\
			&-\operatorname{tr}\left( G_{2}F_{22}\right) -\operatorname{tr}\left(G_{3}\left(F_{11}-F_{12} F_{22}^{-1} F_{12}^{\top}-M\right)\right) \,.
		\end{aligned}
		 \end{equation*}
	 As Problem~\ref{p3} is convex, the necessary condition under which the points are primal and dual optimal points is the Karush-Kuhn-Tucker (KKT) conditions \cite{boyd2004convex}.  That is, if $\hat{F}$, $\hat{M}$, $\hat{G}_{1}$,  $\hat{G}_{2}$ and $\hat{G}_{3}$ are any points that satisfy the KKT conditions of Problem~\ref{p3} summarized below
	
	(1) Primal feasibility: (\ref{dp1-2})--(\ref{dp1-1});
	
	(2) Dual feasibility:
	\begin{equation*}
		G_{1} \succeq \pmb{0}\,, \quad G_{2} \succeq \pmb{0}\,, \quad G_{3} \succeq \pmb{0} \,;
	\end{equation*}
	
	(3) Complementary slackness:
	\begin{equation}\label{cs2}
		\operatorname{tr}\left(G _ { 1 } \left(r \left[\begin{array}{ll}
			A & B
		\end{array}\right]^{\top}\left(F_{11}-F_{12} F_{22}^{-1} F_{12}^{\top}\right) \left[\begin{array}{ll}
			A & B
		\end{array}\right]	-F+W\right) \right) =0\,, 
	\end{equation}
	\begin{equation*}\label{cs3}
		\operatorname{tr}\left(G_{2}F_{22}\right)=0\,,
	\end{equation*}	
	\begin{equation}\label{cs1}
		\operatorname{tr}\left(G_{3}\left(F_{11}-F_{12} F_{22}^{-1} F_{12}^{\top}-M\right)\right)=0\,;
	\end{equation}
	
	(4) Stationarity conditions:
	\begin{equation}\label{sc}
		\frac{\partial\mathscr{L}_{3}}{\partial M} =\pmb{0}\,,\quad  \frac{\partial\mathscr{L}_{3}}{\partial F} =\pmb{0} \,,
	\end{equation}
	then  $\left( \hat{F}, \hat{M}\right) $ and $\left( \hat{G}_{1},  \hat{G}_{2}, \hat{G}_{3}\right)$ are primal and dual optimal points, respectively. 
	
	Choose $\hat{F}=H^{*}$ and $\hat{M}=P^{*}$. Taking $\hat{F}=H^{*}$  into the left-hand side of~(\ref{dp1-2})  causes the left-hand side of~(\ref{dp1-2}) to be equal to $\pmb{0}$, from which it follows that
	 the constraint~(\ref{dp1-2}) is active, i.e., the constraint~(\ref{cs2}) is  satisfied naturally without imposition of any constraints on  $\hat G_{1}$. Moreover, taking  $(\hat{F}, \hat{M})=\left( H^{*}, P^{*}\right) $   into the left-hand side of constraint~(\ref{dp1-1})  causes the left-hand side of~(\ref{dp1-1}) to be equal to $\pmb{0}$.
	That is, the constraint~(\ref{dp1-1}) is  active and~(\ref{cs1}) is satisfied naturally without imposition of any constraints on  $\hat G_{3}$.	
	Since  $\hat{F}_{22}=H^{*}_{22}\succ \pmb{0}$, the constraint~(\ref{dp1-3}) is not active at the  point $\hat{F}=H^{*}$,  which results in    $\hat{G}_{2}=\pmb{0}$.
	
	It is shown below that the stationary conditions ~(\ref{sc}) can be satisfied at $(\hat{F}, \hat{M})$ by a proper choice of $\hat{G}_{1}$ and $\hat{G}_{3}$.
	To this end, the matrix variable $F$ is reformulated as 
	 \begin{equation*}
		F=E_{1} F_{11} E_{1}^{\top}+E_{1} F_{12} E_{2}^{\top}+E_{2} F_{12}^{\top} E_{1}^{\top}+E_{2} F_{22} E_{2}^{\top}	
	\end{equation*}
	 by utilizing $E_{1}= \left[\begin{array}{ll}I_{n} & \pmb{0}\end{array}\right]^{\top}$ and  $E_{2}=\left[\begin{array}{ll}\pmb{0} & I_{m}\end{array}\right]^{\top} $.
	Then, (\ref{sc}) can be equivalently expressed as~(\ref{sc1})--(\ref{sc4}).
	\begin{align}
		\frac{\partial\mathscr{L}_{3}}{\partial M} &=-\left( X_{0}+\frac{r}{1-r} \Sigma\right) ^{\top}+G_{3}^{\top}=\pmb{0}\,, \label{sc1}\\
		\frac{\partial\mathscr{L}_{3}}{\partial F_{11}}  &=-G_{3}-r \left[\begin{array}{ll}
			A & B
		\end{array}\right] G_{1}\left[\begin{array}{ll}
			A & B
		\end{array}\right]^{\top}+E_{1}^{\top} G_{1} E_{1}=\pmb{0}\,, \label{sc2}\\
		\frac{\partial\mathscr{L}_{3}}{\partial F_{12}}  &=2r \left[\begin{array}{ll}
			A & B
		\end{array}\right]G_{1}\left[\begin{array}{ll}
			A & B
		\end{array}\right]^{\top} \hat{F}_{12} \hat{F}_{22}^{-1}+2G_{3} \hat{F}_{12} \hat{F}_{22}^{-1} +2E_{1}^{\top} G_{1} E_{2}=\pmb{0}\,, \label{sc3}\\
		\frac{\partial\mathscr{L}_{3}}{\partial F_{22}}  &=E_{2}^{\top} G_{1} E_{2}-\hat{F}_{22}^{-1}\left(\begin{array}{ll}
			F_{12}^{\top}G_{3} \hat{F}_{12}
		\end{array}+\hat{F}_{12}^{\top}r \left[\begin{array}{ll}
			A & B
		\end{array}\right] G_{1}\left[\begin{array}{ll}
			A & B
		\end{array}\right]^{\top} \hat{F}_{12}\right) \hat{F}_{22}^{-1}=\pmb{0}\,. \label{sc4}
	\end{align}
	Choose 
	\begin{equation}\label{g3}
		\hat{G}_{3}= X_{0}+\frac{r}{1-r} \Sigma \succ \pmb{0}\,, 
	\end{equation}
	which satisfies the stationary condition~(\ref{sc1}).
	
	From~(\ref{sc2}) and (\ref{sc3}),  $G_{1}$ should  be subject to the following constraint:
	\begin{equation}\label{eq25}
		E_{1}^{\top} G_{1} E_{1} \hat{F}_{12} \hat{F}_{22}^{-1}+E_{1}^{\top} G_{1} E_{2}=\pmb{0}. 
	\end{equation}
	Furthermore, from~(\ref{sc2}) and (\ref{sc4}), $G_{1}$ should also be subject to the following constraint:
	\begin{equation}\label{eq26}
		-\hat{F}_{12}^{\top} E_{1}^{\top} G_{1} E_{1} \hat{F}_{12}+\hat{F}_{22} E_{2}^{\top} G_{1} E_{2} \hat{F}_{22}=\pmb{0}\,.
	\end{equation}
	By dividing $G_{1}$ into the block matrix $G_{1}=\left[\begin{array}{ll}G^{1} _{11} & G^{1}_{12} \\ \left( G^{1}_{12}\right) ^{\top} & G^{1}_{22}\end{array}\right]$,   (\ref{eq25}) and (\ref{eq26}) are represented as
	\begin{align}
		&G^{1} _{11} \hat{F}_{12} \hat{F}_{22}^{-1}+G^{1} _{12}=\pmb{0}\,, \label{eq27}\\
		&-\hat{F}_{12}^{\top} G^{1} _{11} \hat{F}_{12}+\hat{F}_{22}G^{1} _{22}\hat{F}_{22}=\pmb{0}\,.\label{eq28}
	\end{align}
	By substituting (\ref{g3})  (\ref{eq27}) and (\ref{eq28}) in (\ref{sc2}), we have
	\begin{equation}\label{eq29}
		\hat G_{3}+r\left[A-B \hat{F}_{22}^{-1} \hat{F}_{12}^{\top}\right]G^{1}_{11}\left[A-B \hat{F}_{22}^{-1} \hat{F}_{12}^{\top}\right]^{\top}=G^{1}_{11} \,.	
	\end{equation}
	Based on the representation of the optimal control gain given in~(\ref{lstar}), substituting  $\hat{F}=H^{*}$ in $A-B \hat{F}_{22}^{-1} \hat{F}_{12}^{\top}$ results in $A+BL^{*}$. Since $L^{*} \in \mathcal{L}$ leads to $\rho\left(A-B \hat{F}_{22}^{-1} \hat{F}_{12}^{\top}\right)<1$.  Then, when  
	\begin{equation*}
		r>1-\frac{\lambda_{\min }\left(\hat G_{3} \right)}{\lambda_{\max }\left(\left[A-B \hat{F}_{22}^{-1} \hat{F}_{12}^{\top}\right]G^{1}_{11}\left[A-B \hat{F}_{22}^{-1} \hat{F}_{12}^{\top}\right]^{\top}\right)}\,,	
	\end{equation*}  
	one can obtain from Lemma 1  that  for  given $\hat G_{3}= X_{0}+\frac{r}{1-r} \Sigma\succ\pmb{0}$, there exists a unique $\hat G^{1}_{11}\succ \pmb{0}$  such that~(\ref{eq29}) holds. 
	Accordingly,  construct $\hat{G}_{1}$  as follows
	\begin{equation*}
		\begin{aligned}
			\hat{G}_{1} & =\left[\begin{array}{cc}
				\hat{G}^{1}_{11} & -\hat{G}^{1}_{11} \hat{F}_{12} \hat{F}_{22}^{-1} \\
				-\hat{F}_{22}^{-1} \hat{F}_{12}^{\top} \hat{G}^{1}_{11} & \hat{F}_{22}^{-1} \hat{F}_{12}^{\top} \hat{G}^{1}_{11} \hat{F}_{12} \hat{F}_{22}^{-1}
			\end{array}\right] \\
			& =\left[\begin{array}{c}
				I_{n} \\
				-\hat{F}_{22}^{-1} \hat{F}_{12}^{\top}
			\end{array}\right] \hat{G}^{1}_{11}\left[\begin{array}{ll}
				I_{n} & -\hat{F}_{12}\left(\hat{F}_{22}^{-1}\right)^{\top}
			\end{array}\right] \succeq  \pmb{0}\,.
		\end{aligned}
	\end{equation*}
	In summary, there  exist  $\hat{F}=H^{*}$,  $\hat{M}=P^{*}$, $	\hat{G}_{1} =\left[\begin{array}{c}
		I_{n} \\
		-\hat{F}_{22}^{-1} \hat{F}_{12}^{\top}
	\end{array}\right] \hat{G}^{1}_{11}\left[\begin{array}{c}
	I_{n} \\
	-\hat{F}_{22}^{-1} \hat{F}_{12}^{\top}
\end{array}\right]^{\top}$, $	\hat{G}_{2} =\pmb{0}$ and $\hat G_{3}= X_{0}+\frac{r}{1-r} \Sigma_0$ that satisfy the KKT conditions of Problem~\ref{p3}, thus $\left( H^{*},P^{*}\right) $ and 
	$\left(\left[\begin{array}{c}
			I_{n} \\
			-\hat{F}_{22}^{-1} \hat{F}_{12}^{\top}
		\end{array}\right] \hat{G}^{1}_{11}\left[\begin{array}{ll}
			I_{n} & -\hat{F}_{12}\left(\hat{F}_{22}^{-1}\right)^{\top}
		\end{array}\right]\,,\pmb{0}\,, X_{0}+\frac{r}{1-r} \Sigma_0 \right)$	
		are primal and dual optimal points, respectively. Hence,  $F^{*}=H^{*}$ and $M^{*}=P^{*}$. Moreover, according to the expressions of $H^{*}$ and $P^{*}$, $\left( H^{*}, P^{*}\right) $ is  independent of $ X_{0}+\frac{r}{1-r} \Sigma$.
\end{proof}
\begin{Remark}\rm	
	Note that for optimization problems that hold for strong duality, when the primal problem is convex, the KKT conditions are  sufficient and necessary conditions such that the points are primal and dual optimal  \cite{boyd2004convex}. And for general primal problems that are convex, strong duality holds under some modest conditions (e.g., the Slater's condition)  \cite{boyd2004convex}. The fact that the primal problem of this paper is non-convex makes the task of proving that strong duality holds nontrivial, as it was in  \cite{lee2018primal} and \cite{li2022model}. Instead of proceeding directly to show that strong duality holds, this paper finds a relationship between the optimal point to the dual problem and the parameter in Q-learning, starting from the dual problem  and then using the convexity of the dual problem and the conclusion that the KKT conditions in convex optimization is a sufficient condition for optimality \cite{boyd2004convex}.
\end{Remark}
Based on the above two theorems, we can now represent the  dual problem  associated with the SLQ problem 
as a standard SDP problem in Problem~\ref{p4}.
\begin{Proposition}\rm
	Problem~\ref{p3} is equivalent to Problem~\ref{p4} in a meaning that optimal points are the same.	
\end{Proposition}
\begin{Problem}\rm \label{p4} SDP with optimization variables $ F \in \mathcal{S}^{n+m}$ and $M \in \mathcal{S}^{n}$.
	\begin{align}
		\underset{F,M }{\max} & \ \operatorname{tr}\left(M\right)\,, \label{p4-1}\\
		\text { s.t. } 	&\left[\begin{array}{cc}
			{r\left[\begin{array}{cc}
					A & B
				\end{array}\right]^{\top}F_{11}\left[\begin{array}{ll}
					A & B
				\end{array}\right]-F+W} & r^{\frac{1}{2}}{\left[\begin{array}{ll}
					A & B
				\end{array}\right]^{\top} F_{12}} \\
			r^{\frac{1}{2}} \left( F_{12}\right) ^{\top}\left[\begin{array}{ll}
				A & B
			\end{array}\right] & F_{22}
		\end{array}\right] \succeq \pmb{0}\,,\label{p4-3}\\
		& \left[\begin{array}{cc}
			F_{11}-M & F_{12} \\
			F_{12}^{\top} & F_{22}
		\end{array}\right] \succeq \pmb{0}\,.	\label{p4-2}
	\end{align} 
\end{Problem} 
\begin{proof}	
	  Using the Schur complement property, the constraints~(\ref{dp1-2}) and~(\ref{dp1-1}) of Problem~\ref{p3} can be rewritten  in the  LMI  forms described by~(\ref{p4-3}) and~(\ref{p4-2}), respectively.  Moreover,  the constraint~(\ref{dp1-3}) of Problem~\ref{p3} is implicitly contained in~(\ref{p4-2}). That is, the constraints of Problem~\ref{p3} and Problem~\ref{p4} are equivalent, then their optimal points are the same, and subsequently we denote them by the same notation $\left(F^{*}, M^{*}\right) $. Finally, Theorem~\ref{th2} means  that   $\left(F^{*}, M^{*}\right) $  is independent of  $\left( X_{0}+\frac{r}{1-r} \Sigma\right) $. Then, without loss of generality, the objective function~(\ref{dp1}) of Problem~\ref{p3} can be expressed as~(\ref{p4-1}) in Problem~\ref{p4}. Thus, Problem~\ref{p3} is equivalently described by the SDP in Problem~\ref{p4}.	
\end{proof}
With the above preparation, we
can now present the  model-based SDP algorithm.
\begin{algorithm}[htb]
	\caption{Model-Based SDP Algorithm}
	\label{alg:0}
	\begin{algorithmic}[1]
		\STATE  Choose a positive discount factor $r$ close to 1 and choose the initial state $\pmb{x}_{0} \sim \mathcal{N}_{n}( \pmb{\mu}_{0},\Sigma_{0})$, where  $\Sigma_{0}  \succ \pmb{0}$.
		\STATE   Solve  SDP in Problem~\ref{p4}  for $\left( F^{*}, M^{*}\right)$.	
		\STATE  Obtain the optimal state feedback controller $ L^{*}=-\left( F^{*}_{22}\right) ^{-1} \left( F^{*}_{12}\right) ^{\top}$.
	\end{algorithmic}
\end{algorithm}	

\section{Model-free  implementation of model-based SDP algorithm } 
The optimization problem in Section 3 involves knowing the system dynamics. In this section, we attempt to avoid this requirement 
and propose the model-free implementation of   SLQ controller design.

Assume  that the   triplets $\left(	\pmb{x}_{k}^{(i)},\pmb{u}_{k}^{(i)},\pmb{x}_{k+1}^{(i)} \right)\,,i =  1,\ldots,l\,,k = 0,\ldots,N-1 $ are available, where $l$ denotes the number of  sampling experiments implemented and  $N$ denotes the length of the sampling time  horizon in the data collection phase.  
Using the data obtained from the 
$l$ sampling experiments to define $Y^{(i)} \triangleq \left[\begin{array}{llll}
	\pmb{x}_{1}^{(i)} & \pmb{x}_{2}^{(i)} & \cdots & \pmb{x}_{N}^{(i)} 
\end{array}\right]$ and $Z^{(i)} \triangleq\left[\begin{array}{llll}
	\pmb{x}_{0}^{(i)} & \pmb{x}_{1}^{(i)} & \cdots & \pmb{x}_{N-1}^{(i)} \\
	\pmb{u}_{0}^{(i)} & \pmb{u}_{1}^{(i)} & \cdots & \pmb{u}_{N-1}^{(i)}
\end{array}\right]$, respectively, $i =  1,\ldots,l$.  Suppose that for each $i =  1,\ldots,l$, $Z^{(i)}$ has full row  rank.  According to the properties of matrix congruence,    the inequality constraint~(\ref{dp1-2}) in Problem~\ref{p3} is left-multiplied by $\left( Z^{(i)}\right) ^{\top}$  and right-multiplied by $Z^{(i)}$ , which yields its equivalent constraint
\begin{equation*}
	r\left( Z^{(i)}\right)^{\top}\left[\begin{array}{ll}
		A & B
	\end{array}\right]^{\top}\left(F_{11}-F_{12} F_{22}^{-1} F_{12}^{\top}\right)\left[\begin{array}{ll}
		A & B
	\end{array}\right] Z^{(i)}-\left( Z^{(i)}\right) ^{\top}\left( F-W\right) Z^{(i)}\succeq \pmb{0}\,.
\end{equation*}	
Using the properties of positive semidefinite  matrices and $l>0$, one obtains 
\begin{equation}\label{qiwang}
	\begin{aligned}
	&\frac{1}{l}\sum_{i=1}^{l}\left[ r\left( Z^{(i)}\right)^{\top}\left[\begin{array}{ll}
		A & B
	\end{array}\right]^{\top}\left(F_{11}-F_{12} F_{22}^{-1} F_{12}^{\top}\right)\left[\begin{array}{ll}
		A & B
	\end{array}\right] Z^{(i)}\right] \\
-&\frac{1}{l}\sum_{i=1}^{l}\left[\left( Z^{(i)}\right) ^{\top}\left( F-W\right) Z^{(i)}\right] \succeq \pmb{0}\,.
\end{aligned}
\end{equation}
Based on the Monte-Carlo method, we can approximate the mathematical expectation using the numerical average of $l$ sample paths. Hence, noting  that  
\begin{equation*}
	\begin{aligned}
		&\mathbb{E}\left[ Y^{\top}\left(F_{11}-F_{12} F_{22}^{-1} F_{12}^{\top}\right) Y \right] \\
		=&\mathbb{E}\left[Z^{\top}\left[\begin{array}{ll}
			A & B
		\end{array}\right]^{\top} \left(F_{11}-F_{12} F_{22}^{-1} F_{12}^{\top}\right)\left[\begin{array}{ll}
			A & B
		\end{array}\right]Z\right]\\
		&+\mathbb{E}\left[\mathcal{W}^{\top}\left(F_{11}-F_{12} F_{22}^{-1} F_{12}^{\top}\right)\mathcal{W}\right], 
	\end{aligned}
\end{equation*}
where
\begin{equation*}
	\begin{aligned}
		 	 Z &\triangleq\left[\begin{array}{llll}
			\pmb{x}_{0} & \pmb{x}_{1} & \cdots & \pmb{x}_{N-1} \\
			\pmb{u}_{0} & \pmb{u}_{1} & \cdots & \pmb{u}_{N-1}
		\end{array}\right]\,, \quad
		\mathcal{W}\triangleq\left[\begin{array}{llll}
			\pmb{w}_{0} & \pmb{w}_{1} & \cdots & \pmb{w}_{N-1}\end{array}\right]\,,\\
		Y& \triangleq \left[\begin{array}{llll}
			\pmb{x}_{1} & \pmb{x}_{2} & \cdots & \pmb{x}_{N} 
		\end{array}\right]
	=\left[\begin{array}{ll}
			A & B
		\end{array}\right] Z+\mathcal{W}\,,
	\end{aligned}
\end{equation*}
  one has 
\begin{equation}\label{p5-33}
	\begin{aligned}
		&\frac{1}{l}\sum_{i=1}^{l}\left[ r\left( Z^{(i)}\right)^{\top}\left[\begin{array}{ll}
			A & B
		\end{array}\right]^{\top}\left(F_{11}-F_{12} F_{22}^{-1} F_{12}^{\top}\right)\left[\begin{array}{ll}
			A & B
		\end{array}\right] Z^{(i)}\right]\\
		\approx&\frac{1}{l}\sum_{i=1}^{l}\left[r\left( Y^{(i)}\right)^{\top}\left(F_{11}-F_{12} F_{22}^{-1} F_{12}^{\top}\right)Y^{(i)}\right]-r\operatorname{tr} \left( \left(F_{11}-F_{12} F_{22}^{-1} F_{12}^{\top}\right)\Sigma\right) I_N \,.
	\end{aligned}
\end{equation}
By substituting (\ref{p5-33})  in  (\ref{qiwang}),  we have
\begin{equation}\label{p5-3}
	\begin{aligned}
	&\frac{1}{l}\sum_{i=1}^{l}\left[r\left( Y^{(i)}\right)^{\top}\left(F_{11}-F_{12} F_{22}^{-1} F_{12}^{\top}\right)Y^{(i)}\right]-r\operatorname{tr} \left( \left(F_{11}-F_{12} F_{22}^{-1} F_{12}^{\top}\right)\Sigma\right) I_N\\
	-&\frac{1}{l}\sum_{i=1}^{l}\left[ \left( Z^{(i)}\right) ^{\top}\left( F-W\right) Z^{(i)}\right] \succeq \pmb{0}\,.
\end{aligned}
\end{equation}
Because  $F_{22}\succ \pmb{0}$, applying Schur  complement and using matrix direct sum, one can equivalently express the constraint~(\ref{p5-3}) as
\begin{equation}\label{p5-66}
	\begin{aligned}
		\left[\begin{array}{cc}
			r\tilde{I} \tilde{D}^{\top}\tilde{F}_{11}  
				\tilde{D}\tilde{I}^{\top}+e & r^{\frac{1}{2}}\tilde{I}\tilde{D} ^{\top} \tilde{F}_{12} \\
			r^{\frac{1}{2}} \tilde{F}_{12} ^{\top}\tilde{D}\tilde{I}^{\top}  &\tilde{F}_{22} 	
		\end{array}\right] \succeq 0\,,	
	\end{aligned}
\end{equation}
\begin{equation*}
	\begin{aligned}
\text {where } \quad	 &e\triangleq-rl\operatorname{tr} \left( \left(F_{11}-F_{12} F_{22}^{-1} F_{12}^{\top}\right)\Sigma\right) I_N-\sum_{i=1}^{l}\left[ \left( Z^{(i)}\right) ^{\top}\left( F-W\right) Z^{(i)}\right]\,,\\
&\tilde{D}\triangleq  \underbrace{Y^{(1)}\oplus Y^{(2)}\oplus\cdots \oplus Y^{(l)}}_{l}\,,\quad \tilde{F}_{11}\triangleq \underbrace{F_{11}\oplus F_{11}\oplus\cdots \oplus F_{11}}_{l}\,,\\
	&\tilde{F}_{12}\triangleq \underbrace{F_{12}\oplus F_{12}\oplus\cdots \oplus F_{12}}_{l},\quad  \tilde{F}_{22}\triangleq \underbrace{F_{22}\oplus F_{22}\oplus\cdots \oplus F_{22}}_{l}\,,\\
	&\tilde{I}\triangleq\left[\underbrace{\begin{array}{llll}
			I_{N}& I_{N} & \cdots & I_{N}
	\end{array}}_{l}\right]\,.
		\end{aligned}
	\end{equation*}

Based on the above analysis, the following Theorem~\ref{th3} gives an SDP that solves the SLQ problem in the presence of unknown information about the system dynamics and without the need for an initial stabilizing control policy. 
\begin{theorem}\rm \label{th3}
	Suppose that $Z^{(i)}$ has full row  rank, for each $i =  1,\ldots,l$.   Then the optimal state feedback gain associated with the SLQ problem is approximated by
	\begin{equation}\label{lo}
		\hat{L}^{*}=-\left(F_{22}^{d}\right)^{-1}\left(F_{12}^{d}\right)^{\top} \,,
	\end{equation}
	where  $\left( F^{d},M^{d}\right)$  is the optimal point of  Problem~\ref{p5} with $F^{d}=\left[\begin{array}{cc}F_{11}^{d} & F_{12}^{d} \\ \left(F_{12}^{d}\right)^{\top} & F_{22}^{d}\end{array}\right]$.
	
	\begin{Problem}\rm	\label{p5}
		SDP with optimization variables  $F \in \mathcal{S}^{n+m}$  and  $M \in \mathcal{S}^{n}$.
		\begin{equation*}
			\begin{aligned}
			\underset{F,M }{\max} & \ \operatorname{tr}\left(M\right)\,, \\
			\text { s.t. } 	&\left[\begin{array}{cc}
			r\tilde{I} \tilde{D}^{\top}\tilde{F}_{11}  
					\tilde{D}\tilde{I}^{\top}+e 
				& r^{\frac{1}{2}}\tilde{I}\tilde{D}^{\top} \tilde{F}_{12} \\
				r^{\frac{1}{2}} \tilde{F}_{12}^{\top}\tilde{D}\tilde{I}^{\top} & \tilde{F}_{22} 	
			\end{array}\right]\succeq \pmb{0}\,,\\
			& \left[\begin{array}{cc}
				F_{11}-M & F_{12} \\
				F_{12}^{\top} & F_{22}
			\end{array}\right] \succeq \pmb{0}\,.
		\end{aligned}
	\end{equation*}

	\end{Problem}
\end{theorem} 
\begin{proof}
	It is already shown that according to the properties
	of positive semidefinite  matrices and matrix congruence, (\ref{p5-66}) is equivalent to  (\ref{dp1-2}) in Problem~\ref{p3}. 
	According to Theorem~\ref{th2} and Proposition~2, the parameters of the optimal Q-function can be obtained by solving Problem~\ref{p3}, which is equivalent to Problem~\ref{p5}. Based on Lemma~2, the parameters of the optimal Q-function lead to an estimate of the optimal state feedback controller,  as shown in~(\ref{lo}).
\end{proof}
\begin{Remark}\rm
	In order to ensure that $Z^{(i)}$ has full row  rank, for each $i =  1,\ldots,l$, the length of the sampling time  horizon  $N$ must be more than or equal to the sum of the dimensions of the input and the state.  Moreover, the rank condition in Theorem~\ref{th3} is analogous to the persistent excitation (PE) requirement \cite{willems2005note}
	in some sense. Both of these conditions are designed to allow the system to oscillate sufficiently during the learning process to produce enough data to be able to fully learn the system information. In practice, it is necessary to add a probing noise consisting of a sum of sinusoids or Gaussian white noise to the control to make sure that both conditions hold.
	\end{Remark}
In view of the results given above, Algorithm~\ref{alg:1} provides a way to solve the SLQ problem based on the available data.
\begin{algorithm}[htb]
	\caption{Model-Free  SDP Algorithm}
	\label{alg:1}
	\begin{algorithmic}[1]
		\STATE  Choose a positive discount factor $r$ close to 1 and choose the initial state $\pmb{x}_{0} \sim~\mathcal{N}_{n}( \pmb{\mu}_{0},\Sigma_{0})$, where  $\Sigma_{0}  \succ \pmb{0}$.
		\STATE Apply  control policy  with probing noises to system~(\ref{eq1}) and collect the  data to construct the data matrices $Z^{(i)}$ and $Y^{(i)}$  for each $i =  1,\ldots,l$ until the  condition in Theorem~\ref{th3} is satisfied.		
		\STATE   Solve SDP in Problem~\ref{p5}   for $\left( F^{d},M^{d}\right)$.	
		\STATE  Obtain the  estimate of the optimal  control gain $ \hat{L}^{*} $ associated with SLQ problem  described by~(\ref{lo}).
	\end{algorithmic}
\end{algorithm}	
\begin{Remark}\rm
	It should be noted that the model-free  primal–dual  (MF-PD) algorithm    in  \cite{li2022model} needs an initial stabilizing controller, while our proposed model-free  SDP algorithm does not need an initial stabilizing control policy. In addition, the MF-PD  algorithm    in  \cite{li2022model} is an iterative algorithm. Theoretically, the state-feedback gain reaches its optimal value only when the number of iterations approaches infinity. While our proposed algorithm is  non-iterative. Furthermore, our proposed algorithm does not need hyper-parameter tuning. 	
\end{Remark}
\section{Simulation}
In this section, the proposed model-free SDP algorithm is tested on different systems, consisting of both stable and unstable dynamic systems. In Section 5.1, the effectiveness of the proposed algorithm is verified in the design of an SLQ controller for a turbocharged diesel engine equipped of exhaust gas recirculation. In Section 5.2, a numerical example of an open-loop unstable system is considered. 
\subsection{Example 1: A Stable System}
In this section, we consider the design of an SLQ controller for a turbocharged diesel engine equipped of exhaust gas recirculation. The discrete-time dynamical system is obtained by discretizing the continuous-time one given in \cite{jung2005comparison,jiang2012computational} which is in the form of $\dot{\pmb{x}}  =A_{c} \pmb{x}+B_{c} \pmb{u} $ with 
\begin{equation*}
	A_{c}=\left[\begin{array}{cccccc}
		-0.4125& -0.0248& 0.0741& 0.0089& 0& 0\\
		101.5873& -7.2651& 2.7608& 2.8068& 0& 0\\
		0.0704& 0.0085& -0.0741& -0.0089& 0& 0.0200\\
		0.0878& 0.2672& 0& -0.3674& 0.0044& 0.3962\\
		-1.8414& 0.0990& 0& 0& -0.0343& -0.0330\\
		0& 0& 0 &-359.0000& 187.5364& -87.0316		
	\end{array}\right]\,,
\end{equation*}
\begin{equation*}
B_{c}=\left[\begin{array}{cc}
		-0.0042& 0.0064\\
		-1.0360& 1.5849\\
		0.0042& 0\\
		0.1261& 0\\
		0& -0.0168\\
		0& 0\end{array}\right] 
\end{equation*}
via standard zero-order-hold discretization techniques with sampling time $T = 0.1s$. To
save space, we do not present the coefficient matrices $A$, $B$ here. 
Select  the weight matrices and discount factor  as $Q =\operatorname{diag} (1, 1, 0.1, 0.1, 0.1, 0.1) $, $R = I_2$  and $r = 0.8$, respectively.  Assume that the discrete-time dynamical system 
suffers from additive noise $\pmb{w}_{k}$ and that  $\pmb{w}_{k}\sim\mathcal{N}_{6}\left(0, 0.001I_6\right)$.

In the proceeding, the model-free SDP Algorithm~\ref{alg:1} is utilized to design the optimal controller.
Choose the initial state $\pmb{x}_{0}$ from a Gaussian distribution with   mean vector $\pmb{\mu}_{0}=[1, 2, 3,4,5,6]^{\top}$ and  covariance matrix $\Sigma_{0}= 5I_6 \succ \pmb{0} $ and then apply the input with Gaussian exploration noise to  system~(\ref{eq1}) to generate  $l  = 100000$ state trajectories in the sampling time  horizon $\left\lbrace 0,\ldots,9 \right\rbrace $  to construct the matrices $Z^{(i)}$ and $Y^{(i)}$  for each $i =  1,\ldots,l$.  Problem~\ref{p5} is solved using CVX  for MATLAB  by Mosek  as the core
solver  and the optimal  gain  is acquired  as 
\begin{equation*}
	\hat{L}^{*}=\left[\begin{array}{cccccc}
		0.7448&    0.0328&    0.0270 &  -0.0056&    0.0147&   0\\
		-1.3009&   -0.0528&   -0.0497&   -0.0405 &   0.0015&   -0.0002
	\end{array}\right]\,.
\end{equation*}
Fig.~1  depicts the state trajectories of the closed-loop systems during the data
collection  phase of the reinforcement learning process. 
\begin{figure}[htb]
	\begin{center}
		\includegraphics[width=\textwidth,height=4cm]{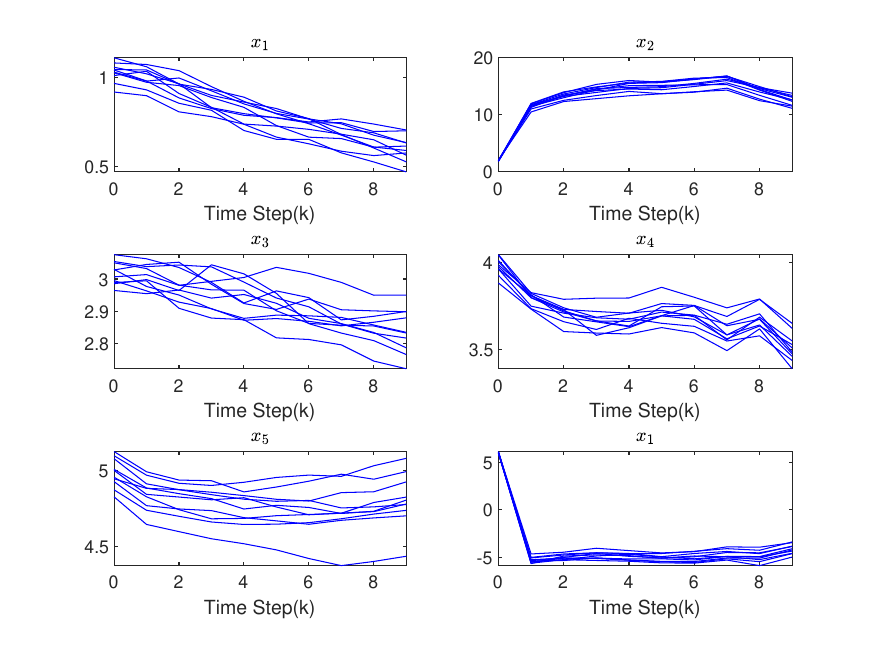}    
		\caption{The state trajectories during the data collection phase.}  
		\label{fig1}                                 
	\end{center}                                 
\end{figure}
Fig.~2
illustrates the numerical average of the state trajectories   using the    data obtained from $100000$ sampling trials  during the model-free design approach, where the trajectories in time domains $\left\lbrace 0,\ldots,9 \right\rbrace $ are the numerical average of  $100000$  trajectories obtained during the data collection phase, at time step $10$, the optimal data-driven controller is implemented and the trajectories are optimally regulated. 
\begin{figure}[htb]
	\begin{center}
		\includegraphics[width=0.8\textwidth,height=4cm]{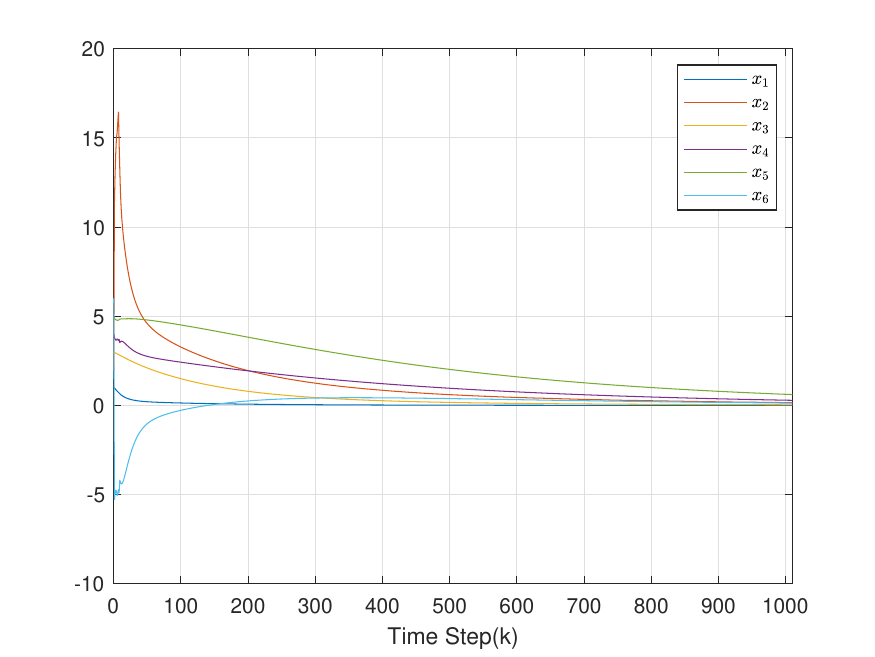}    
		\caption{A numerical average of state trajectories in data collection phase and after learning.}  
		\label{fig2}                                 
	\end{center}                                 
\end{figure}  

In the following, the model-free SDP Algorithm 2 is compared to model-based algorithm. In the presence of no structural constraints, the  solution of DARE~(\ref{are1}) reveals realizable optimal point, which can be seen as a baseline for the simulation results. When the system matrices $A$ and $B$ are available, the optimal control gain can be obtained using the MATLAB command DARE as 
\begin{equation*}
	L^{*}=\left[\begin{array}{cccccc}
		0.7446&    0.0328&    0.0271 &  -0.0057&    0.0146&   0\\
		-1.3011&   -0.0528&   -0.0496&   -0.0404 &   0.0015&   -0.0002
	\end{array}\right]\,.
\end{equation*} 
Comparing these two optimal control gains obtained and subtracting them, we have
\begin{equation*}
	L^{*}-\hat{L}^{*}=10^{-3}\times\left[\begin{array}{cccccc}
		
		-0.1912&    0.0085&    0.0995&   -0.0618&   -0.0493&    0.0052\\
		-0.2288&    0.0093&    0.0614&    0.0260&    0.0012&    0.0038
	\end{array}\right]
\end{equation*}
and  $\left\|L^{*}-\hat{L}^{*}\right\|_{F}=10^{-4}\times3.3123$. 
The above comparison results show that although Algorithm~\ref{alg:1} does not know the information about the system dynamics, the optimal control gain obtained is very close to that obtained from the known model information, and the model-free Algorithm~\ref{alg:1} performs very effectively. 

\subsection{Example 2: An Unstable System}
Consider system~(\ref{eq1}) with
\begin{equation*}
	A=\left[\begin{array}{cccc}
		0.3& 0.4& 0.2& 0.2\\
		0.2& 0.3& 0.3& 0.2\\
		0.2& 0.2& 0.4& 0.4\\
		0.4& 0.2& 0.2& 0.4
	\end{array}\right]\,,\quad B= I_4\,.
\end{equation*}
Select  the weight matrices and discount factor  as 
\begin{equation*}
	Q=\left[\begin{array}{cccccc}
		3& -1& 0& -1\\
		-1& 3& -1& 0\\
		0& -1& 3& -1\\
		-1& 0& -1& 3		
	\end{array}\right]\,, 
\end{equation*}
 $R = I_4$  and $r = 0.8$, respectively.  Assume that  system~(\ref{eq1})
suffers from additive noise $\pmb{w}_{k}$ and that  $\pmb{w}_{k}\sim\mathcal{N}_{4}\left(0, 0.01I_4\right)$.
 The open-loop system is unstable because the maximum eigenvalue is 1.1267. 

The model-free Algorithm~\ref{alg:1} is employed to develop the optimal controller in the steps that follow. Choose  $\pmb{x}_{0}\sim~\mathcal{N}_{n}( \pmb{\mu}_{0},\Sigma_{0})$, where   $\pmb{\mu}_{0}=[1, 2, 3,4]^{\top}$ and   $\Sigma_{0}= 5I_4 \succ \pmb{0} $ and then apply the input with Gaussian exploration noise to  system~(\ref{eq1}) 
to produce $l  = 100$  state trajectories in the sampling time horizon $\left\lbrace 0,\ldots,8 \right\rbrace $   to build the matrices $Z^{(i)}$ and $Y^{(i)}$  for each $i =  1,\ldots,l$.  CVX is used to solve Problem~\ref{p5}  and the optimal control gain is acquired as
\begin{equation*}
	\hat{L}^{*}=\left[\begin{array}{cccc}
		-0.1757 &  -0.2506&   -0.1089&   -0.1038\\
		-0.1085&   -0.1744&   -0.1743&   -0.1027\\
		-0.1046&   -0.1087&   -0.2503&   -0.2453\\
		-0.2530&   -0.1062&   -0.1061&   -0.2475
	\end{array}\right]\,.
\end{equation*}
Fig.~3 depicts the state trajectories derived from the numerical average of the data collected from $100$ sampling experiments during the model-free design approach, where the trajectories in time domains $\left\lbrace 0,\ldots,8 \right\rbrace $ are the state trajectories gathered during the data collection phase of the reinforcement learning process, and starting at moment $9$, the system begins to apply the model-free optimal controller to generate optimal state trajectories.
\begin{figure}[htb]
	\begin{center}
		\includegraphics[width=0.8\textwidth,height=4cm]{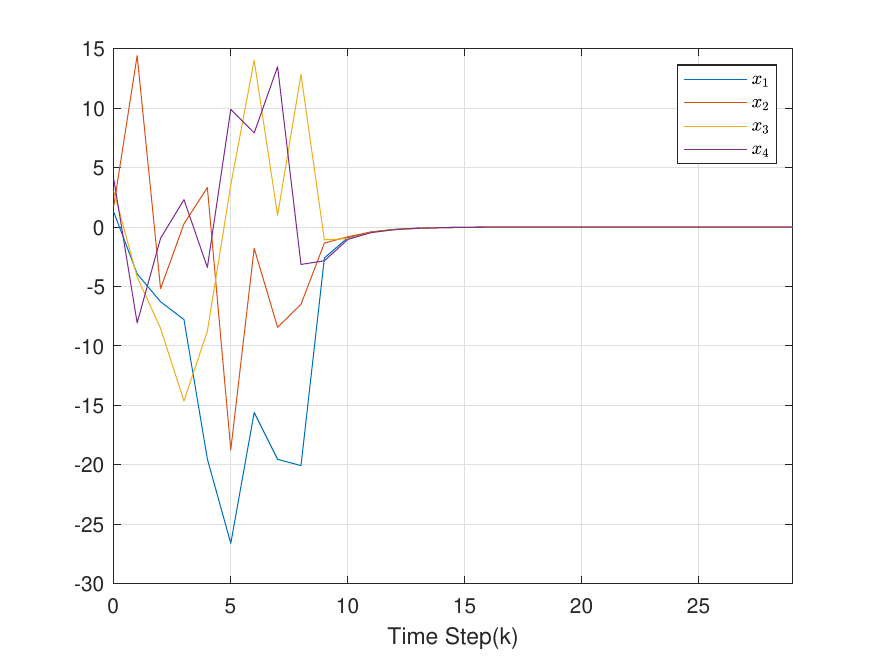}    
		\caption{A numerical average of state trajectories in data collection phase and after learning.}  
		\label{fig3}                                 
	\end{center}                                 
\end{figure}

Utilizing the system model to solve DARE~(\ref{are1}), the optimal controller is obtained as
\begin{equation*}
	L^{*}=\left[\begin{array}{cccc}
		-0.1756 &  -0.2506&   -0.1089&   -0.1039\\
		-0.1084&   -0.1744&   -0.1743&   -0.1028\\
		-0.1046&   -0.1087&   -0.2504&   -0.2453\\
		-0.2529&   -0.1062&   -0.1062&   -0.2476
	\end{array}\right]\,.
\end{equation*} 
Comparing these two optimal control gains obtained and subtracting them, one has
\begin{equation*}
	L^{*}-\hat{L}^{*}=10^{-4}\times
	\left[\begin{array}{cccc}
		0.4984&    0.0699&   -0.0886&   -0.2940\\
		0.5387&   -0.1775&   -0.4655&   -0.4758\\
		0.2987&   -0.0016&   -0.1178&   -0.2228\\
		0.7544&    0.1269&   -0.1820&   -0.4037
	\end{array}\right]
\end{equation*}
and  $\left\|L^{*}-\hat{L}^{*}\right\|_{F}=10^{-4}\times1.4306$. Comparison results indicate that model-free Algorithm~\ref{alg:1} exhibits good performance in terms of accuracy and effectiveness even for unstable systems. 
\section{Conclusion}
In this paper, a novel model-free SDP method for solving
the SLQ problem  has been proposed. The controller design scheme is based on  an SDP  with LMI
constraints and  is sample-efficient and non-iterative. Moreover, it does not necessitate the use of  an initial stabilizing controller.  In addition, the proposed algorithm  has been tested on a turbocharged diesel engine equipped of exhaust gas recirculation and an unstable numerical example. The results confirm that the proposed algorithm is able to  efficiently
learn the optimal control gains and stabilize the systems. Future work  focus on extending the proposed method to stochastic systems with both additive and multiplicative noise by using the proposed framework and performing stability analysis in the model-free design.

\bibliographystyle{elsarticle-num}
\bibliography{gjref.bib}







\end{document}